\documentclass[a4paper]{amsart}
\usepackage{graphicx}
\usepackage{amssymb}
\usepackage{amsmath}
\usepackage{amsthm}
\usepackage{amscd}
\usepackage[all,2cell]{xy}

\UseAllTwocells \SilentMatrices
\newtheorem{thm}{Theorem}[section]

\newtheorem{cor}[thm]{Corollary}
\newtheorem{lem}[thm]{Lemma}
\newtheorem{exm}[thm]{Example}

\newtheorem{prop}[thm]{Proposition}
\theoremstyle{definition}
\newtheorem{defn}[thm]{Definition}
\theoremstyle{remark}
\newtheorem{rem}[thm]{\bf Remark}
\numberwithin{equation}{section}

\begin{document}
\title[Irreducible representations of Leavitt path algebras]
{Irreducible representations of Leavitt path algebras}

\author[Xiao-Wu Chen] {Xiao-Wu Chen}

\thanks{The author is supported by Special Foundation of President of The Chinese Academy of Sciences
(No.1731112304061) and  National Natural Science Foundation of China (No.10971206).}
\date{\today}

\thanks{E-mail: xwchen$\symbol{64}$mail.ustc.edu.cn}

\maketitle

\dedicatory{}%
\commby{}%

\begin{abstract}
We construct some irreducible representations of the Leavitt path algebra of an arbitrary
quiver. The constructed representations are associated to certain algebraic branching systems. For a row-finite quiver, we
classify algebraic branching systems, to which irreducible representations of the Leavitt path algebra are associated. For
a certain quiver, we obtain a faithful completely reducible representation of the Leavitt path
algebra. The twisted representations of the constructed ones under the scaling action are studied.
\end{abstract}

\section{Introduction}

Let $k$ be a field and let $Q$ be an arbitrary quiver. The notion of the path algebra $kQ$ of $Q$
is well known in  representation theory (\cite{ARS}).  Unlike this, the Leavitt path algebra
$L_k(Q)$ of $Q$ with coefficients in $k$ is relatively new, which is introduced in \cite{AA05} and \cite{AMP}. Leavitt path algebras generalize the important
algebras studied by Leavitt in \cite{Lev57, Lev}, and they are given as algebraic analogues
of the Cuntz-Krieger $C^*$-algebras $C^*(Q)$ (\cite{CK, Ra05}). Recent research indicates
that the Leavitt path algebra of a quiver captures certain homological properties  of both the path algebra
and its Koszul dual; see \cite{AB, Smi, Ch}.

The representation theory of the Leavitt path algebra $L_k(Q)$ is studied in \cite{Ar, AB, GR}. In \cite{AB}, the authors
 prove that the category of finitely presented $L_k(Q)$-modules is equivalent to a quotient category of
  the corresponding category of $kQ$-modules. This result is extended in \cite{Smi} via a completely different method.
  Using the notion of algebraic branching system,  a construction of $L_k(Q)$-modules is given in \cite{GR}. Moreover, some sufficient conditions are given to guarantee the faithfulness of the constructed modules.

   We are interested in simple modules, or equivalently, irreducible representations
   of Leavitt path algebras. Recall that irreducible representations that can be embedded in the
   Leavitt path algebra itself are just minimal left ideals. These representations are classified in 
   \cite{AMMS08, AMMS}. This classification plays an important role in the study of the socle series 
   of Leavitt path algebras; see \cite{ARS11}.

    In this paper,  we construct some irreducible representations of the Leavitt path algebra $L_k(Q)$ of an arbitrary quiver $Q$. These representations are divided into two classes. We use infinite paths of the quiver  to construct representations in the first class, and use finite paths that terminate at a sink for the second class. Our construction is inspired by a construction  of representations of Cuntz algebras   in \cite{LNS}. In view of  results in \cite{Smi}, it seems that the representations in the first class relate to the point modules studied in \cite{Smii}, while the latter
   plays an important role in non-commutative algebraic geometry. The representations in the second class are isomorphic to minimal left ideals of $L_k(Q)$ that is generated by idempotents corresponding to sinks of the quiver $Q$; these representations are known, at least for countable quivers (\cite{AMMS08, AMMS}).

   The paper is structured as follows. We recall some basic notions and introduce some terminologies in Section 2. The main construction is given in Section 3; see Theorems \ref{thm:1} and \ref{thm:N}. In Section 4, we draw some  consequences from the constructed representations. Based on results in \cite{AMMS}, we point out that for a countable quiver, the constructed irreducible representations contain all minimal left ideals of the Leavitt path algebra; see Proposition \ref{prop:minimal}. We prove the faithfulness of a certain completely reducible representation; see Proposition \ref{prop:faithfulness}. Section 5 is devoted to relating the constructed representations to algebraic branching systems in \cite{GR}. For a row-finite quiver, we classify  algebraic branching systems whose associated representations are irreducible. It turns out that irreducible representations associated to algebraic branching systems are necessarily isomorphic to the ones constructed in Section 3; see Theorem \ref{thm:irre}. In the final section, we study the twisted representations of the constructed irreducible representations under the scaling action. This allows us to obtain new irreducible representations and prove another faithfulness result of some completely reducible representation; see Theorem \ref{thm:conclusion} and Proposition \ref{prop:faithfulness2}.

\section{Preliminaries}
We recall basic notions related to quivers and Leavitt path algebras, and introduce some terminologies for later use.
The references for quivers are \cite[Chapter III]{ARS} and \cite{AA05}, and for Leavitt path algebras are
\cite{AA05, AMP, AA08, Tom}.

\subsection{Quivers and left-infinite paths}

Recall that a \emph{quiver}  $Q=(Q_0, Q_1; s,t)$ consists of a set $Q_0$ of vertices, a set $Q_1$ of arrows
and two maps $s, t\colon Q_1\rightarrow Q_0$, which assign an arrow $\alpha$ to its starting and terminating vertices
 $s(\alpha)$ and $t(\alpha)$, respectively. A quiver is also called a directed graph. A vertex where there is no arrow
  starting is called a \emph{sink}, and a vertex where there are infinitely many arrows starting is called an \emph{infinite emitter}.  A vertex is \emph{regular} if it is neither a sink nor an infinite emitter. The quiver $Q$ is \emph{regular} (resp. \emph{row-finite})  provided  that each vertex is regular (resp. not an infinite emitter).

 A (finite) \emph{path} in the quiver $Q$  is
 a sequence $p=\alpha_n \cdots \alpha_2\alpha_1$ of arrows with $t(\alpha_i)=s(\alpha_{i+1})$ for $1\leq i\leq n-1$; in this
 case, the path $p$ is said to have length $n$, denoted by $l(p)=n$. We denote $s(p)=s(\alpha_1)$ and $t(p)=t(\alpha_n)$. We identify an arrow with a path of length one, and associate to each vertex $i$  a trivial path $e_i$ of length zero. A nontrivial path $p$ with the same starting and terminating vertex is an \emph{oriented cycle}. An oriented cycle of length one is called a \emph{loop}.

Let $k$ be a field.  We denote by $Q_n$ the set of paths of length $n$, and by $kQ_n$ the vector space over $k$
 with basis $Q_n$. Here, we identify a vertex $i$ with the corresponding trivial path $e_i$.
 The \emph{path algebra} is defined as $kQ=\oplus_{n\geq 0}kQ_n$, whose multiplication is given as follows:
  for two paths $p$ and $q$, if $s(p)=t(q)$, then the product $pq$ is the concatenation of paths; otherwise,
  set the product $pq$ to be zero. We write the concatenation of paths from the right to the left.

  The path algebra $kQ$ is naturally $\mathbb{N}$-graded. Observe that  for a vertex $i$ and a  path $p$, $pe_{i}=\delta_{i, s(p)} \; p$ and $e_i p= \delta_{i, t(p)} \; p$.  Here, $\delta$ is the Kronecker symbol. In particular, $\{e_i\; |\; i\in Q_0\}$ is a set
  of pairwise orthogonal idempotents in $kQ$. Observe that the $k$-algebra $kQ$ is not necessarily unital.

 We need infinite paths in a quiver. A \emph{left-infinite path} in $Q$ is an infinite sequence
 $p=\cdots \alpha_n\cdots \alpha_2\alpha_1$ of arrows with $t(\alpha_i)=s(\alpha_{i+1})$ for all $i\geq 1$. Set $s(p)=s(\alpha_1)$. Denote by $Q_{\infty}$ the set of left-infinite paths in $Q$.  For example, for an oriented cycle $q$,  we have a left-infinite path $q^\infty=\cdots q\cdots qq$; such a left-infinite path is said to be \emph{cyclic}. We remark that the set $Q_\infty$ together
   with the product topology plays an important role in symbolic dynamic system (\cite{LM}).

 For a left-infinite path $p$ and $n\geq 1$, denote by $\tau_{\leq n}(p)=\alpha_n\cdots \alpha_2\alpha_1$ and  $\tau_{>n}(p)=\cdots \alpha_{n+2}\alpha_{n+1}$ the two truncations. Observe that $\tau_{\leq n}(p)$ lies in $Q_n$ and that $\tau_{> n}(p)$ is a left-infinite path. Hence, a left-infinite path $p$ is cyclic if and only if there exists some  $n\geq 1$ such that $p=\tau_{>n}(p)$. We set $\tau_{\leq 0}(p)=e_{s(p)}$ and $\tau_{>0}(p)=p$.

 Two left-infinite paths $p$ and $q$ are \emph{tail-equivalent}, denoted by $p\sim q$, provided that
there exist $n$ and $m$ such that $\tau_{>n}(p)=\tau_{>m}(q)$; compare \cite[1.4]{Smii}. This is an equivalence relation on $Q_{\infty}$. We denote
by $\widetilde{Q}_\infty$ the set of tail-equivalence classes, and for a path $p$ denote the corresponding class by $[p]$.

A left-infinite path $p$ is \emph{rational} provided that there exists  $n\geq 1$ such that $p\sim \tau_{>n}(p)$.
This is equivalent to that $p$ is tail-equivalent to a cyclic path.  In this case, we have that $p\sim q^\infty$ for a \emph{simple} oriented cycle $q$. Here, an oriented cycle is simple if it is not a power of a shorter oriented cycle. Otherwise, the path $p$ is called \emph{irrational}. This is equivalent to that for each pair $(n, m)$ of distinct natural numbers, we have $\tau_{>n}(p)\neq \tau_{>m}(p)$.

If a left-infinite path $p$ is rational (resp. irrational), then the corresponding class $[p]$ is called  a \emph{rational class} (resp. an \emph{irrational class}); such classes form a subset $\widetilde{Q}^{\rm rat}_\infty$ (resp. $\widetilde{Q}^{\rm irr}_\infty$) of $\widetilde{Q}_\infty$. Then we have a disjoint union $\widetilde{Q}_\infty=\widetilde{Q}^{\rm rat}_\infty \cup \widetilde{Q}_\infty^{\rm irr}$.

\subsection{Leavitt path algebras}
  Let $Q$ be a quiver and $k$ a field. Consider the set of formal symbols $\{\alpha^*\; |\; \alpha\in Q_1\}$. The \emph{Leavitt path algebra} $L_k(Q)$ of $Q$ with coefficients in $k$ is a $k$-algebra given by generators $\{e_i, \alpha, \alpha^*\; |\; i\in Q_0, \alpha\in Q_1\}$ subject to the following relations:
\begin{enumerate}
\item[(0)] $e_ie_j=\delta_{ij}\; e_i$ for all $i\in Q_0$;
\item[(1)] $e_{t(\alpha)}\alpha=\alpha=\alpha e_{s(\alpha)}$ for all $\alpha\in Q_1$;
\item[(2)] $e_{s(\alpha)}\alpha^*=\alpha^*=\alpha^* e_{t(\alpha)}$ for all $\alpha\in Q_1$;
\item[(3)] $\alpha \beta^*=\delta_{\alpha, \beta} \; e_{t(\alpha)}$ for all $\alpha, \beta\in Q_1$;
\item[(4)] $\sum_{\{\alpha\in Q_1\; |\; s(\alpha)=i\}} \alpha^*\alpha=e_i$ for all regular vertices $i\in Q_0$.
\end{enumerate}

The relations (3) and (4) are called the \emph{Cuntz-Krieger relations}. Here, we emphasize that  $k$-algebras
are not required to be unital.

We observe that $L_k(Q)$ is naturally $\mathbb{Z}$-graded such that ${\rm deg}\; e_i=0$, ${\rm deg}\; \alpha=1$ and
${\rm deg}\; \alpha^*=-1$. There is a natural graded algebra homomorphism $ \iota_Q\colon kQ\rightarrow L_k(Q)$ such that $\iota_Q(e_i)=e_i$ and $\iota_Q(\alpha)=\alpha$.  Here, we abuse the notation:  for a path $p\in kQ$ we denote its image $\iota_Q(p)$ still by $p$. Moreover, if $p=\alpha_n\cdots \alpha_2\alpha_1$, we set $p^*=\alpha_1^*\alpha_2^*\cdots \alpha_n^*$ in $L_k(Q)$. For convention, $e_i^*=e_i$ for $i\in Q_0$. The algebra homomorphism $\iota_Q$ is injective; see \cite[Lemma 1.6]{G09} or Proposition \ref{prop:embedding}.

The following fact is immediate from the the relation (3). Observe that for finite paths $p, q$ in $Q$, $p^*q=0$ if
$t(q)\neq t(p)$.

\begin{lem}\label{lem:multi} {\rm (\cite[Lemma 3.1]{Tom})} Let $p, q, \gamma$ and $\eta$ be finite paths in $Q$ with
$t(p)=t(q)$ and $t(\gamma)=t(\eta)$. Then in $L_k(Q)$ we have
$$(p^*q)(\gamma^*\eta)= \left\{ \begin{array} {r@{}l}
&(\gamma' p)^* \eta  \quad \mbox{ if } \gamma=\gamma' q; \\
& p^* \eta  \quad \quad \; \;  \mbox{ if } q=\gamma; \\
& p^*(q'\eta) \quad \mbox{ if } q=q'\gamma; \\
& 0 \quad \quad \quad \; \; \mbox{ otherwise.}
\end{array} \right.$$
Here, $\gamma'$ and $q'$ are some nontrivial  paths in $Q$. \hfill $\square$
\end{lem}

We have the following immediate consequence; see \cite[Lemma 1.5]{AA05} or \cite[Corollary 3.2]{Tom}.

\begin{cor}\label{cor:1} The Leavitt path algebra $L_k(Q)$ is spanned by the
following set
\begin{center} \qquad \qquad \qquad $\{p^*q\; |\; p, q \mbox{ are finite paths in } Q \mbox{ with } t(p)=t(q) \}.$
\hfill $\square$ \end{center}
\end{cor}

By Corollary \ref{cor:1}, a nonzero element $u$ in $L_k(Q)$ can be written in its \emph{normal form}
\begin{align}\label{equ:1}
u=\sum_{i=1}^l \lambda_i\;  p_i^*q_i,
\end{align}
 where $l\geq 1$, each $\lambda_i\in k$ is nonzero, and $p_i$'s and $q_i$'s are paths in $Q$ with
$t(p_i)=t(q_i)$. We require in addition that the pairs $(p_i, q_i)$ are pairwise distinct. The normal
form in general is not unique; see the relation (4).

 Inspired by the paragraphs following \cite[Lemma 1.7]{AA05},
we define a numerical invariant $\kappa(u)$ of $u$ as the smallest natural number $n_0$ such that in one of its normal forms
$u=\sum_{i=1}^l \lambda_i \; p_i^*q_i$, we have $l(p_i)\leq n_0$ for all $i$. For example, $\kappa(u)=0$ if and only if
$u$ can be written as $u=\sum_{i=1}^l \lambda_i q_i$ for some paths $q_i$, if and only if
$u$ lies in the image of $\iota_Q$; compare \cite[Definition 3.3]{Tom}.

The Leavitt path algebra $L_k(Q)$ in general is not unital. However, the set $\{e_i\; |\; i\in Q_0\}$ of pairwise
orthogonal idempotents is a set of \emph{local units} in the following sense: for a nonzero element $u=\sum_{i=1}^l \lambda_i \; p_i^*q_i$ in its normal form, set $x=\sum_{\{j\in Q_0\; |\; j=s(p_i) \mbox{ \small for some } i\}} e_j$ and $y=\sum_{\{j\in Q_0\; |\; j=s(q_i) \mbox{ \small for some }i\}} e_j$, then we have $u=xuy$. In particular, there exists some $j\in Q_0$ such that $e_ju\neq 0$. For details, we refer to \cite[Lemma 1.6]{AA05} or \cite[3.2]{Tom}.

\section{The construction of irreducible representations}

In this section, we construct two classes  of irreducible representations of Leavitt path algebras.
These constructed irreducible representations are  pairwise non-isomorphic.

\subsection{The representation $\mathcal{F}$}

Let $k$ be  a field, and let $Q$ be a quiver. We denote by $\mathcal{F}$ the vector space over $k$
with basis given by the set $Q_\infty$ of left-infinite paths in $Q$. For each tail-equivalence class $[p]$
in $\widetilde{Q}_\infty$, denote by $\mathcal{F}_{[p]}$ the subspace of $\mathcal{F}$ spanned by the set $\{q\; |\;
q\in [p]\}$.  Then we have $\mathcal{F}=\oplus_{[p]\in \widetilde{Q}_{\infty}} \mathcal{F}_{[p]}$.

We will construct a representation of the Leavitt path algebra $L_k(Q)$ on $\mathcal{F}$. We point out
that our construction is inspired by a construction in the proof of \cite[Theorem II]{LNS}.

 For each vertex $i\in Q_0$, define a linear map $P_i\colon \mathcal{F}\rightarrow \mathcal{F}$ such
that $$P_i(p)=\delta_{i, s(p)}\;  p$$
 for all $p\in Q_\infty$.

 For each arrow $\alpha\in Q_1$, define a linear map
$S_\alpha\colon \mathcal{F}\rightarrow \mathcal{F}$ such that
$$S_\alpha(p)=\delta_{\alpha, \alpha_1} \tau_{>1}(p)$$
for $p=\cdots \alpha_2\alpha_1\in Q_\infty$. We define another linear map $S^*_\alpha\colon \mathcal{F}\rightarrow \mathcal{F}$ such that $$S^*_\alpha(p)=\delta_{t(\alpha), s(p)}\; p\alpha=\delta_{t(\alpha), s(\alpha_1)}\; p\alpha.$$ Here, we recall by definition that $s(p)=s(\alpha_1)$.

\begin{prop}\label{prop:1}
There is an algebra homomorphism $\rho\colon L_k(Q)\rightarrow {\rm End}_k(\mathcal{F})$ such that $\rho(e_i)=P_i$,
$\rho(\alpha)=S_\alpha$ and $\rho(\alpha^*)=S^*_\alpha$ for all $i\in Q_0$ and $\alpha\in Q_1$.
\end{prop}

\begin{proof}
To see the existence of such a homomorphism, it suffices to show that the linear maps $P_i$, $S_\alpha$ and $S^*_\alpha$ satisfy
the defining relations of the Leavitt path algebra.

For (0), we observe that $P_i \circ P_j=\delta_{ij} P_i$.

For (1), we have that for $p=\cdots \alpha_2\alpha_1\in Q_\infty$, $$P_{t(\alpha)}S_\alpha(p)=\delta_{t(\alpha),s(\alpha_2)}\; \delta_{\alpha, \alpha_1} \tau_{>1}(p)=\delta_{\alpha, \alpha_1}\tau_{>1}(p)=S_\alpha(p).$$
 Here, we use that if $\alpha=\alpha_1$, then $t(\alpha)=s(\alpha_2)$. Similarly, we have
$S_\alpha \circ P_{s(\alpha)}=S_\alpha$. Similar arguments prove the relation (2).

For (3), we have that  $$S_\alpha S^*_\beta(p)= \delta_{\alpha, \beta} \delta_{t(\beta), s(p)}  \; \tau_{>1}(p\beta)=\delta_{\alpha, \beta} \delta_{t(\alpha), s(p)}\;  p=\delta_{\alpha, \beta} P_{t(\alpha)}(p).$$

For (4), we have that $$\sum_{\{\alpha\in Q_1\; |\; s(\alpha)=i\}} S^*_\alpha S_\alpha (p)=\sum_{\{\alpha\in Q_1\; |\; s(\alpha)=i\}} S^*_\alpha (\delta_{\alpha, \alpha_1}\;  \tau_{>1}(p)) = \delta_{i, s(p)}\; p =P_i(p).$$

 Then we are done.
\end{proof}

We denote the action of  $L_k(Q)$ on $\mathcal{F}$ by ``.", that is, $a.u=\rho(a) (u)$ for $a\in L_k(Q)$ and $u\in \mathcal{F}$.  For a nonzero element $u$ in $\mathcal{F}$, its \emph{normal form} means
the expression $u=\sum_{i=1}^l \lambda_i p_i$,  where each $\lambda_i\in k$ is nonzero and
the left-infinite paths $p_i$ are pairwise distinct.

The following result yields the first class of irreducible representations. In particular, the representation
$\mathcal{F}$ turns out to be completely reducible.

\begin{thm}\label{thm:1}
Consider the representation $\mathcal{F}$ of $L_k(Q)$. Then the following statements hold.
\begin{enumerate}
\item[(1)] For each $[p]\in \widetilde{Q}_\infty$, the subspace $\mathcal{F}_{[p]}\subseteq \mathcal{F}$ is an irreducible
sub representation, which satisfies that ${\rm End}_{L_k(Q)}(\mathcal{F}_{[p]})\simeq k$.
\item[(2)] Two representations $\mathcal{F}_{[p]}$ and $\mathcal{F}_{[q]}$ are isomorphic if and only if $[p]=[q]$.
\end{enumerate}
\end{thm}

\begin{proof} To see that $\mathcal{F}_{[p]}\subseteq \mathcal{F}$ is a sub representation, it suffices to notice that
for each left-infinite path $p$, $p\sim \tau_{>1}(p)$ and $p\sim p\alpha$ for all arrows $\alpha$ with $t(\alpha)=s(p)$.

To prove that the representation $\mathcal{F}_{[p]}$ is irreducible, take a nonzero sub representation $U\subseteq \mathcal{F}_{[p]}$, and a nonzero element $u=\sum_{i=1}^l \lambda_ip_i$ in $U$.  Here, the expression of $u$ is its normal form.
Take $n$ large enough such that all the $\tau_{\leq n}(p_i)$'s  are pairwise distinct. Then
we have $\tau_{\leq n}(p_1).u=\tau_{\leq n} (p_1). (\lambda_1 p_1)=\lambda_1 \; \tau_{>n}(p_1)$.  This proves that $p_0=\tau_{>n}(p_1)$ lies in $U$. We claim that each $p'\in [p]$ lies
in $U$. Then we are done. We observe  that $p'\sim p_0$. Assume that $\tau_{>r}(p')=\tau_{>s}(p_0)$. The equalities
$\tau_{>s}(p_0)=\tau_{\leq s}(p_0).p_0$ and $p'=(\tau_{\leq r}(p'))^*.\tau_{>r}(p')$ force that $p'$ lies in $U$.

Consider a nonzero homomorphism $\phi\colon \mathcal{F}_{[p]}\rightarrow \mathcal{F}_{[q]}$. Since $\mathcal{F}_{[p]}$
is irreducible, $\phi$ is injective. Let $p'\in [p]$ and write $\phi(p')=\sum_{i=1}^l \lambda_i q_i$ in its normal form.  We claim that $l=1$ and $q_1=p'$. Otherwise, we may assume that $q_1\neq p'$.
Take $n$ large enough such that   all the $\tau_{\leq n}(q_i)$'s  are pairwise distinct and that $x=\tau_{\leq n}(q_1)\neq \tau_{\leq n}(p')$. Then $x.p'=0$ and $x.\phi(p')=x.(\lambda_1 q_1)=\lambda_1 \; \tau_{>n}(q_1)\neq 0$. A contradiction!

The above claim proves (2). Moreover, we have that a nonzero endomorphism $\phi\colon \mathcal{F}_{[p]}\rightarrow \mathcal{F}_{[p]}$ necessarily satisfies that $\phi(p')=\lambda_{p'} \; p'$ with $\lambda_{p'}\in k$ for all $p'\in [p]$. It remains to see that all the $\lambda_{p'}$'s  are the same, and then we have ${\rm End}_{L_k(Q)}(\mathcal{F}_{[p]})\simeq k$. Take $p'$ and $p''$ in $[p]$. We assume that  $\tau_{>r}(p')=\tau_{>s}(p'')$. We deduce from the equalities
$\tau_{>s}(p'')=\tau_{\leq s}(p'').p''$ and $p'=(\tau_{\leq r}(p'))^*.\tau_{>r}(p')$  that $\lambda_{p'}=\lambda_{p''}$.
\end{proof}

\begin{exm}\label{exm:Leavitt}
 {\rm Let $n\geq 1$ and let $Q=R_n$ be the quiver consisting of one vertex and $n$ loops attached to it. Then the Leavitt
path algebra $L_k(R_n)$ is the \emph{Leavitt algebra} of order $n$ (\cite{Lev57, Lev}).

Consider the case $n=1$. The algebra $L_k(R_1)$ is isomorphic to the Laurant polynomial algebra $k[x, x^{-1}]$. Here, the set $Q_\infty$ consists of a single element, and then the representation $\mathcal{F}$ is irreducible. In fact,  $\mathcal{F}$ is one-dimensional, on which $x$ acts as the identity.

Consider the case $n\geq 2$. Then the set $\widetilde{Q}_\infty$ of tail-equivalence classes is uncountable. So we obtain a uncountable family of irreducible representations $\mathcal{F}_{[p]}$ for the Leavitt algebra.}
\end{exm}

\subsection{The representation $\mathcal{N}$}

Let $k$ be a field and  let $Q$ be a quiver. Denote by $Q_0^s$ the set consisting of all sinks in $Q$.  Denote by $\mathcal{N}$
 the vector space over $k$ with basis given by all the finite paths in $Q$ that terminate at a sink. For each sink $i$,
denote by $\mathcal{N}_i$ the subspace of $\mathcal{N}$ spanned by paths $p$ with $t(p)=i$. Then we have
$\mathcal{N}=\oplus_{i\in Q^s_0}\;  \mathcal{N}_i$.

We will define a representation of $L_k(Q)$ on $\mathcal{N}$. The construction is similar to the one in
the previous subsection.

 For each vertex $i\in Q_0$, define a linear
map $P_i\colon \mathcal{N}\rightarrow \mathcal{N}$ such that $$P_i(p)=\delta_{i, s(p)} \; p$$
for finite paths $p$ terminating at some sink.

For each arrow
$\alpha\in Q_1$, define a linear map $S_\alpha\colon \colon \mathcal{N}\rightarrow \mathcal{N}$
as follows:
$$S_\alpha(p)=0 \mbox{ if } l(p)=0, \mbox{ and }  S_\alpha(p)=\delta_{\alpha, \alpha_1} \; \alpha_n\cdots \alpha_2$$
for $p=\alpha_n \cdots \alpha_2\alpha_1$. We define another linear map $S^*_\alpha\colon \mathcal{N}\rightarrow \mathcal{N}$
such that
  $$S^*_\alpha(p)= \delta_{t(\alpha), s(p)} \; p\alpha=\delta_{t(\alpha), s(\alpha_1)} \; p\alpha.$$

\begin{prop}
There is an algebra homomorphism $\psi \colon L_k(Q)\rightarrow {\rm End}_k(\mathcal{N})$ such that $\psi(e_i)=P_i$,
$\psi(\alpha)=S_\alpha$ and $\psi(\alpha^*)=S^*_\alpha$ for all $i\in Q_0$ and $\alpha\in Q_1$.
\end{prop}

\begin{proof}
The proof is similar to the one of Proposition \ref{prop:1}. We note that in verifying the relation (4), we use that $P_i(e_j)=0$
for $i$ regular and $j\in Q_0^s$.
\end{proof}

The following result gives us the second class of irreducible representations of the Leavitt path
algebra. In particular, the representation $\mathcal{N}$ turns out to be completely reducible.

\begin{thm}\label{thm:N}
Consider the representation $\mathcal{N}$ of $L_k(Q)$. Then the following statements hold.
\begin{enumerate}
\item[(1)] For each $i\in Q_0^s$, the subspace $\mathcal{N}_{i}\subseteq \mathcal{N}$ is an irreducible
sub representation, which satisfies that ${\rm End}_{L_k(Q)}(\mathcal{N}_{i})\simeq k$.
\item[(2)] Two representations $\mathcal{N}_{i}$ and $\mathcal{N}_{j}$ are isomorphic if and only if $i=j$.
\item[(3)] For any $[p]\in \widetilde{Q}_\infty$ and $i\in Q_0^s$, $\mathcal{F}_{[p]}$ is not isomorphic to $\mathcal{N}_i$.
\end{enumerate}
\end{thm}

\begin{proof}
The subspace $\mathcal{N}_i\subseteq \mathcal{N}$ is clearly a sub representation, and it is generated by
the trivial path $e_i$.

For the
irreducibility of $\mathcal{N}_i$, take a nonzero sub representation $U\subseteq \mathcal{N}_i$ and  a nonzero element $u=\sum_{j=1}^l \lambda_j p_j\in U$ in its normal form. That is, each $\lambda_j\in k$ is nonzero, and all the $p_j$'s are pairwise distinct  satisfying that $t(p_j)=i$.  We choose the normal form   such that $p_1$ is longest among all the $p_j$'s (such $p_1$
need not be unique). Then we have $p_1.u=\lambda_1 e_i$. This forces that $e_i\in U$, from which we infer $U=\mathcal{N}_i$. Here, we use ``." to denote the action of
$L_k(Q)$ on $\mathcal{N}$.

Take a nonzero homomorphism $\phi\colon \mathcal{N}_i\rightarrow \mathcal{N}_j$, which is necessarily injective.
Write  $\phi(e_i)=\sum_{r=1}^l \lambda_r p_r$ in its normal form. We claim that $l=1$ and $p_1=e_i$. This  will force that $i=j$ and ${\rm End}_{L_k(Q)}(\mathcal{N}_i)\simeq k$. For the claim,
we assume the converse. Then we may assume that $p_1$ is longest among all the $p_r$'s. In particular, we have $l(p_1)\geq 1$.
Then $p_1.e_i=0$ and $p_1.\phi(e_i)=\lambda_1 e_j\neq 0$. A contradiction!

For (3), it suffices to show that each homomorphism  $\phi\colon \mathcal{N}_i\rightarrow \mathcal{F}_{[p]}$
satisfies that $\phi(e_i)=0$, and then $\phi=0$. Otherwise,  write the nonzero element $\phi(e_i)=\sum_{j=1}^l \lambda_j p_j$  in its normal form. Here, all the $p_j$'s lie in $[p]$. Take $n$ large enough such that all the truncations $\tau_{\leq n}(p_j)$'s  are pairwise distinct. Then $\tau_{\leq n}(p_1).e_i=0$ and $\tau_{\leq n}(p_1). \phi(e_i)=\lambda_1 \; \tau_{>n}(p_1)\neq 0$. This is absurd.
\end{proof}

\begin{rem}
We will show that the irreducible representations $\mathcal{N}_i$ are isomorphic to minimal left ideals of the Leavitt path algebra; see Proposition \ref{prop:minimal}(2).
\end{rem}

\section{Some  consequences, minimal left ideals and a faithfulness result}

In this section, we draw some consequences from  the constructed representations $\mathcal{F}$ and
$\mathcal{N}$. We show that for a countable quiver, the constructed irreducible representations contain all minimal left
  ideals of the Leavitt path algebra. We prove that for a certain quiver, the completely reducible representation $\mathcal{F}\oplus \mathcal{N}$ is faithful.

\subsection{Some consequences} The following result extends slightly a result contained in the proof of \cite[Theorem 5.4]{Smi}. We point out that the injectivity of $\iota_Q$ is known; see \cite[Lemma 1.6]{G09}.

\begin{prop}\label{prop:embedding}
Let $Q$ be an arbitrary  quiver. Then for  $n, m\geq 0$, the following subset of $L_k(Q)$
$$\{p^*q\; |\; p, q \mbox{ are paths in } Q \mbox{ with }t(p)=t(q), l(p)=m \mbox{ and } l(q)=n\}$$
is linearly independent. In particular, the algebra homomorphism $\iota_Q\colon kQ\rightarrow L_k(Q)$ is injective.
\end{prop}

\begin{proof}
The second statement is an immediate consequence of the first one, once we notice that
the homomorphism $\iota_Q$ preserves the gradings, and that $\{q\; |\; l(q)=n\}$ is a basis of
$kQ_n$. Here, we use that $\iota_Q(q)=e_{t(q)}^*q$.

We will show that each element $u=\sum_{i=1}^l \lambda_i \; p_i^*q_i$ in $L_k(Q)$  is nonzero, where each $\lambda_i \in k$ is nonzero and
the pairs $(p_i, q_i)$ are pairwise distinct.  We require that each $(p_i)^*q_i$ is in the above subset. Consider
the terminating vertex $t(q_1)$ of $q_1$. Then we are in two cases. In the first case, there is a path $p$ with $s(p)=t(q_1)$ and $t(p)$ a sink. Consider the element $pq_1$ in $\mathcal{N}_{t(p)}$. We have that $$u. (pq_1)=\sum_{\{i\;|\; 1\leq i\leq l, q_i=q_1 \}} \lambda_i \; pp_i.$$
 Observe that
the  paths $pp_i$'s in the summation are pairwise distinct. Then $u.(pq_1)\neq 0$, which forces that $u\neq 0$.  In
the second case, there is a left-infinite path $p$ with $s(p)=t(q_1)$. Consider the element $pq_1$ in $\mathcal{F}_{[p]}$. Then the same argument as in the first case will work.
\end{proof}

The following observation in the finite case is implicitly contained in \cite[Section 3]{AAS}. Recall that for
a quiver $Q$, $Q_0^s$ denotes the set of all sinks in $Q$.

\begin{prop}\label{prop:sink}
Let $Q$ be an arbitrary quiver. Then the following subset of $L_k(Q)$
$$\{p^*q\; |\; p, q \mbox{ are finite paths in } Q \mbox{ with } t(p)=t(q)\in Q_0^s\}$$
is linearly independent. \end{prop}

\begin{proof}
It suffices to show that each element $u=\sum_{i=1}^l \lambda_i \; p_i^*q_i$ in $L_k(Q)$  is nonzero, where each $\lambda_i \in k$ is nonzero and the pairs $(p_i, q_i)$ are pairwise distinct.  We require that each $(p_i)^*q_i$ is in the above subset. Assume that
$q_1$ is the shortest among the $q_i$'s (such $q_1$ need not be unique). Consider the element $q_1\in \mathcal{N}$. Then we have $$u.q_1=\lambda_1p_1+\sum \lambda_i p_i,$$ where the summation runs over
$2 \leq i\leq l$ with $q_i=q_1$. Observe that these $p_i$'s are different
from $p_1$. It follows that $u.q_1\neq 0$. This proves
that $u$ is nonzero.
\end{proof}

\subsection{Minimal left ideals} We show that some of the irreducible representations constructed
in Section 3 are isomorphic to minimal left ideals of the Leavitt path algebra. For this, we recall some terminologies
from \cite{AMMS08, AMMS}. Let $Q$ be a quiver. A vertex $i$ is called \emph{linear} provided that there is at most
one arrow starting at $i$ and there is no oriented cycles going through $i$.  A linear vertex $i$ is a \emph{line point}
if each vertex $j$, to which there is a (unique) path starts from $i$, is linear.

There are two cases for a line point. A line point $i$ is \emph{infinite} provided that there is a left-infinite path $p$ starting at $i$; this unique path is called the \emph{tail} of the infinite line point. A line point $i$ is \emph{finite}
provided that there is a sink $i_0$ to which there is a path starting from $i$; this unique sink is called the \emph{end}
of the finite line point. We remark that a sink is a finite line point, whose end is itself.

For a vertex $i$ of $Q$, we consider the left ideal $L_k(Q)e_i$ generated by the idempotent $e_i$. This left ideal is viewed as a representation of $L_k(Q)$; it is clearly nonzero by Proposition \ref{prop:embedding}.

\begin{prop}\label{prop:minimal} Let $Q$ be a quiver. Then the following statements hold.
\begin{enumerate}
\item Let $i$ be an infinite line point with tail $p$. Then there is an isomorphism of representations
$$L_k(Q)e_i\simeq \mathcal{F}_{[p]}.$$
\item Let $i$ be a finite line point with end $i_0$. Then there is an isomorphism of representations
$$L_k(Q)e_i\simeq \mathcal{N}_{i_0}.$$
\end{enumerate}
\end{prop}

We consider a \emph{countable} quiver $Q$, that is, both the sets of vertices and arrows are countable.
By \cite[Theorem 4.13]{AMMS}, up to isomorphism, all minimal left ideals of $L_k(Q)$ are of the form
$L_k(Q)e_i$ for some line point $i$. Therefore, the irreducible representations constructed in Section 3 contain
all minimal left ideals of $L_k(Q)$. It seems that a similar result holds for an arbitrary quiver; see \cite[Proposition 1.9 and Theorem 1.10]{ARS11}.

\begin{proof} (1) For each left-infinite path $q$ in $[p]$, take $n(q)\geq 0$ smallest such that
$\tau_{>n(q)}(q)=\tau_{>m}(p)$ for some $m\geq 0$; such $m=m(q)$ is unique, since the tail of an infinite line point is not cyclic.
We observe that for each pair $(n, m)$ such that $\tau_{>n}(q)=\tau_{>m}(p)$, we have $(\tau_{\leq n}(q))^*\tau_{\leq m}(p)=(\tau_{\leq n(q)}(q))^*\tau_{\leq m(q)}(p)$ in $L_k(Q)$; here, we use the second Cuntz-Krieger relation and the fact that each vertex appearing in $p$ is linear.

We define a linear map $\mathcal{F}_{[p]} \rightarrow L_k(Q)e_i$, sending $q$ to $(\tau_{\leq n(q)}(q))^*\tau_{\leq m(q)}(p)$.
It is a homomorphism of representations by direct verification. Since the homomorphism sends $p$ to $e_i$, by the irreducibility
of $\mathcal{F}_{[p]}$ we deduce that it is an isomorphism.

(2) Let $q$ be the unique path from $i$ to its end $i_0$. Then we have an isomorphism $L_k(Q)e_{i_0}\rightarrow L_k(Q)e_i$ sending
$x$ to $xq$; compare \cite[Lemma 2.2]{AMMS08}. The inverse is given by the multiplication of $q^*$ from the right. Here, we apply the Cuntz-Krieger relations to have $qq^*=e_{i_0}$ and $q^*q=e_i$.

We define a linear map $\mathcal{N}_{i_0}\rightarrow L_k(Q)e_{i_0}$ sending $p$ to $p^*$. It sends $e_{i_0}$ to $e_{i_0}=e_{i_0}^*$. The map is a homomorphism of representations by direct verification. Then it follows from the irreducibility of $\mathcal{N}_{i_0}$ that the map is an isomorphism.
\end{proof}

\subsection{A faithfulness result}Recall that a quiver $Q$ is row-finite, provided that there is no infinite emitter in $Q$. A left-infinite path $p$
which is not cyclic is said to be  \emph{non-cyclic}.  This is equivalent to that $p\neq \tau_{>n}(p)$ for any $n\geq 1$.

We point out that a part of the argument in the following proof resembles the one given in the first step in the proof of
\cite[Theorem 4.2]{GR}.

\begin{prop}\label{prop:faithfulness}
Let $Q$ be a row-finite quiver.  Assume that for each vertex $i$ in $Q$, there exists either a path which starts at $i$ and terminates at a sink, or a non-cyclic left-infinite path which starts at $i$.  Then the representation $\mathcal{F}\oplus \mathcal{N}$ is faithful.
\end{prop}

\begin{proof}
We will show that for each nonzero element  $u\in L_k(Q)$, its action on $\mathcal{F}\oplus \mathcal{N}$ is nonzero.
Write $u=\sum_{i=1}^l \lambda_i\;  p_i^*q_i$ in its normal form; see (\ref{equ:1}). Moreover, there exists $j\in Q_0$ such that $e_ju\neq 0$;  see Subsection 2.2. Observe that if the action of $e_j u$ on $\mathcal{F}\oplus \mathcal{N}$  is
nonzero, so does $u$. So we replace $u$ by $e_j u$. This amounts to the requirement that in the normal form of $u$,
$s(p_i)=j$ for all $i$.

We use induction on the numerical invariant $\kappa(u)$ introduced in Subsection 2.2. For the case $\kappa(u)=0$, we have that $u=\sum_{i=1}^l \lambda_i q_i$ and $t(q_i)=j$. Without loss of generality, we assume that $q_1$ is shortest among all the $q_i$'s.
Consider the vertex $j$.  Then we are in two cases. In the first case, there is a path $p$ with $s(p)=j$ and $t(p)$ a sink.   Then $(pu). (pq_1)=\lambda_1 e_{j}\neq 0$. Here, we view  $pq_1\in \mathcal{N}$. This shows that $pu$ acts nontrivially  on $\mathcal{N}$, and so does $u$.

In the second case,  there is a non-cyclic  left-infinite path $p$ with $s(p)=j$.  We view $pq_1\in \mathcal{F}$.  Then we have $u.(pq_1)=\sum_{i=1}^l \lambda_i\;  q_i.(pq_1)$. Observe that for $i\neq 1$, $q_i.(pq_1)\neq 0$ if and only if $q_i=\tau_{\leq n_i} (p)q_1$ with
$n_i=l(q_i)-l(q_1)$, in which case we have that $q_i.(pq_1)=\tau_{>n_i}(p)$ and $n_i\geq 1$. Consequently, we have $$u.(pq_1)=\lambda_1 p+\sum \lambda_i \; \tau_{>n_i}(p),$$
 where the summation runs over all $i\neq 1$ such that $q_i=\tau_{\leq n_i} (p)q_1$. Since the left-infinite path $p$ is non-cyclic, in particular, $p\neq \tau_{> m}(p)$ for any $m\geq 1$, we have $u.(pq_1)\neq 0$.  This implies that $u$ acts nontrivially on $\mathcal{F}$.

 For the general case, we assume that $\kappa(u)>0$. This implies that $j=s(p_i)$ is not a sink. By assumption,
 the vertex $j$ is not an infinite emitter, and then the vertex $j$ is regular. By the relation (4), we have  $u=e_ju=\sum_{\{\alpha\in Q_1\; |\; s(\alpha)=j\}}\alpha^* \alpha u$. In particular, there is
 an arrow $\alpha$ with $v=\alpha u\neq 0$. Observe by the relation (3) that $\kappa(v)<\kappa(u)$. Hence by
 the induction hypothesis, the action of $v$ on $\mathcal{F}\oplus \mathcal{N}$ is nonzero. This forces
 that the action of  $u$ is also nonzero.
\end{proof}

\begin{rem}\label{rem:faithful}
The conditions on the quiver are somehow necessary for the proposition. Consider the quiver $Q=R_1$ in Example \ref{exm:Leavitt}, that is, it consists of a vertex with one loop attached. The representation $\mathcal{F}$ is one dimensional, and $\mathcal{N}$ is
clearly zero. Then the representation $\mathcal{F}\oplus \mathcal{N}$ is not faithful.
\end{rem}

We illustrate Proposition \ref{prop:faithfulness} with an example.

\begin{exm}\label{exm:faithful}
{\rm
Let $Q$ be a finite quiver without oriented cycles. Then the Leavitt path algebra $L_k(Q)$ is finite dimensional
by Corollary \ref{cor:1}.  In this case, the representation $\mathcal{F}$ is zero. Then the finite dimensional representation
$\mathcal{N}$ is  faithful and  completely reducible. It follows that the Leavitt path algebra $L_k(Q)$ is semi-simple and $\{\mathcal{N}_i |\; i\in Q_0 \mbox{ is a sink}\}$ is a complete set of pairwise non-isomorphic irreducible representations of $L_k(Q)$, each of which has its endomorphism algebra isomorphic to $k$.

We apply the Wedderburn-Artin theorem for semisimple algebras to infer that $L_k(Q)$ is isomorphic to a product of full matrix algebras over $k$; see \cite[Proposition 3.5]{AAS}. We point out that this result can be proved directly by  combining Lemma \ref{lem:multi} and  Proposition \ref{prop:sink}.
}\end{exm}


\section{Algebraic branching systems}

In this section, we relate the irreducible representations constructed in Section 3  to certain
algebraic branching systems in \cite{GR}. This somehow is expected by the authors; see  the second paragraph
in \cite[p.259]{GR}. For a row-finite quiver, we classify algebraic branching systems, whose associated representations of the Leavitt path algebra are irreducible. It turns out that all these irreducible representations are isomorphic to the ones in Section 3.

Let $Q$ be an arbitrary quiver. Following \cite[Definition 2.1]{GR},  a $Q$-\emph{algebraic branching system}
consists of a set $X$, and a family of its subsets $\{X_i, X_\alpha\; |\; i\in Q_0, \alpha\in Q_1\}$ together with
a bijection $\sigma_\alpha\colon X_{t(\alpha)}\rightarrow X_\alpha$  for each arrow $\alpha$, where the subsets are
subject to the following  constraints:
\begin{enumerate}
\item[(1)] $X_i\cap X_j=\emptyset= X_\alpha\cap X_\beta$ for $i\neq j$, $\alpha\neq \beta$;
\item[(2)] $X_\alpha\subseteq X_{s(\alpha)}$ for each $\alpha\in Q_1$;
\item[(3)] $X_i=\bigcup_{\{\alpha\in Q_1 \; |\; s(\alpha)=i\}} X_\alpha$ for each regular vertex $i$.
\end{enumerate}
We will denote the above $Q$-algebraic branching system simply by $X$. We point out that this notion  is closely related to dynamic systems with partitions studied in \cite{Bon}.

A $Q$-algebraic branching system $X$ is \emph{saturated} provided that $X=\cup_{i \in Q_0}X_i$; it is said to be \emph{perfect}, if in addition that the equation in (3) holds for each non-sink vertex of $Q$. For a row-finite quiver $Q$, any saturated $Q$-algebraic branching system is perfect.

Let $X$ and $Y$ be two $Q$-algebraic branching systems. A map $f\colon X\rightarrow Y$ is a \emph{morphism} of $Q$-algebraic branching systems, if $f(X_i)\subseteq Y_i$ and $f(X_\alpha)\subseteq Y_\alpha$ for all vertices $i$ and arrows $\alpha$ of $Q$, and  $f$ is compatible with the bijections inside $X$ and $Y$. Two  $Q$-algebraic branching systems are \emph{isomorphic} provided that there exist mutually inverse morphisms between them.

Examples of $Q$-algebraic branching systems are given in \cite[Theorem 3.1]{GR}. We are interested in
the following examples, both of which are perfect.

\begin{exm}\label{exm}
{\rm (1) Let $p$ be a left-infinite path in $Q$. Consider its tail-equivalence class $[p]$ as a set. It is
a $Q$-algebraic branching system in the following manner: $[p]_{i}=\{q\in [p]\: |\; s(q)=i\}$, $[p]_\alpha=\{q\in [p]\; |\;
q \mbox{ starts with } \alpha \}$. The bijection $\sigma_\alpha \colon [p]_{t(\alpha)} \rightarrow [p]_\alpha$ sends $q$ to $q\alpha$.

(2) Let $i\in Q_0^s$ be a sink. Consider the set $N_i$ consisting of paths in $Q$ that terminate at $i$. It is
a $Q$-algebraic branching system in a similar manner as above.
}
\end{exm}

We recall that one may associate a representation of the Leavitt path algebra   to each $Q$-algebraic branching
system.  Let $X$ be a $Q$-algebraic branching system. Denote by $\mathcal{M}(X)$ the vector space consisting
of all functions from $X$ to $k$, which vanish on all, but  finitely many, elements in $X$. For each $x\in X$, denote
by $\chi_x \colon X\rightarrow k$ the \emph{characteristic  function}. That is, $\chi_x(y)=\delta_{x, y}$ for
 all $x$ and $y$ in $X$. Then $\{\chi_x\; |\; x\in X\}$ is a basis of
$\mathcal{M}(X)$.

The following construction is contained in \cite[Theorem 2.2 and Remark 2.3]{GR}. We adapt the notation
for our convenience.

\begin{lem}
Let $X$ be a $Q$-algebraic branching system. Then there is a representation  of $L_k(Q)$ on $\mathcal{M}(X)$ as follows:
 \begin{enumerate}
 \item[(1)] for each $i\in Q_0$, $e_i.\chi_x=\chi_x$ if $x\in X_i$, otherwise $e_i. \chi_x=0$;
 \item[(2)]   for each $\alpha\in Q_1$, $\alpha.\chi_x=\chi_{\sigma_\alpha^{-1}(x)}$ if $x\in X_\alpha$, otherwise $\alpha. \chi_x=0$;
  \item[(3)] for each $\alpha\in Q_1$, $\alpha^*. \chi_x=\chi_{\sigma_\alpha(x)}$ if $x\in X_{t(\alpha)}$, otherwise  $\alpha^*. \chi_x=0$. \hfill $\square$
      \end{enumerate}
\end{lem}

For a $Q$-algebraic branching system $X$, the above representation $\mathcal{M}(X)$ of $L_k(Q)$ is said to be the \emph{associated representation}
to $X$. Observe that $X$ is saturated if and only if the associated representation $\mathcal{M}(X)$ is unital, that is, $L_k(Q).\mathcal{M}(X)=\mathcal{M}(X)$.

Let $f\colon X\rightarrow Y$ a morphism of $Q$-algebraic branching systems. Assume that $X$ is perfect. Then $f$ induces a homomorphism of  associated representations
$$\mathcal{M}(f)\colon \mathcal{M}(X)\longrightarrow \mathcal{M}(Y),$$
which sends $\chi_x$ to $\chi_{f(x)}$.  Here, we use the fact that $f^{-1}(Y_i)=X_i$ and $f^{-1}(Y_\alpha)=X_\alpha$ for each vertex $i$ and arrow $\alpha$ of $Q$, which is derived directly from the perfectness of $X$. The homomorphism $\mathcal{M}(f)$ is an isomorphism if and only if so is $f$.

The following observation shows that the representations constructed in Section 3 are associated
to the $Q$-algebraic branching systems in Example \ref{exm}.

\begin{prop}\label{prop:FN} Let $Q$ be a quiver. Use the notation as above. Then there are isomorphisms of
representations $$\mathcal{F}_{[p]}\simeq \mathcal{M}([p]) \mbox{ and } \mathcal{N}_i\simeq \mathcal{M}(N_i)$$ for each left-infinite
path $p$ and sink $i$.
\end{prop}

\begin{proof}
The linear map $\mathcal{F}_{[p]}\rightarrow \mathcal{M}([p])$ sending $q$ to $\chi_q$ is an isomorphism of
representations.  This is done by direct verification. The same map works for $\mathcal{N}_i$.
\end{proof}

We infer from Section 3 and Proposition \ref{prop:FN} that the representations associated to algebraic branching systems
in Example \ref{exm} are irreducible.  In some cases, these are all the irreducible representations constructed in this way.

\begin{thm}\label{thm:irre}
Let $Q$ be a quiver and $X$ a perfect $Q$-algebraic branching system. Then the associated representation $\mathcal{M}(X)$ is irreducible if and only if $X$ is isomorphic to $[p]$ or $N_i$, where $p$ is a left-infinite path and $i$ is a sink in $Q$.
\end{thm}

This result implies that for a row-finite quiver $Q$, all the irreducible representations associated to some saturated $Q$-algebraic branching systems are isomorphic to the ones in Section 3.

The following example shows that the perfectness condition in the above theorem is somehow necessary.

\begin{exm}
{\rm Let $Q$ be the following quiver consisting of two vertices $\{1, 2\}$ and infinitely many arrows from
$1$ to $2$.
$$1 \stackrel{\infty}\longrightarrow 2$$
Consider the $Q$-algebraic branching system $X=\{*\}$ consisting of a single element, such that $X_1=X$, $X_2=\emptyset=X_\alpha$ for each arrow $\alpha$. Then X is saturated but not perfect; thus it is isomorphic to none of the $Q$-algebraic branching systems in Example \ref{exm}. However, the associated representation $\mathcal{M}(X)$ is one-dimensional and then irreducible. We refer to \cite[Lemma 1.2]{AA08} for the structure of the Leavitt path algebra $L_k(Q)$.
}\end{exm}

We make some preparation for the proof of Theorem \ref{thm:irre}. The argument here resembles the one in
the proof of  \cite[Theorem 1]{Bon}. Let $X$ be a perfect $Q$-algebraic branching system, and let $x\in X$. If $x\in X_i$ for a non-sink $i$, then there exists a unique arrow $\alpha$ such that $s(\alpha)=i$ and $x\in X_{\alpha}$; thus there exists a unique $y\in X_{t(\alpha)}$ such that $\sigma_\alpha(y)=x$. We repeat this argument to $y$.
Then we infer that  for each element $x\in X$ there are two cases as follows.

In the first case, there exists a unique left-infinite path $p(x)=\cdots \alpha_n\cdots \alpha_2\alpha_1$, such that there exist $x_m\in X_{s(\alpha_{m+1})}$ for $m\geq 0$, satisfying  that $x=x_0$ and $\sigma_{\alpha_m}(x_m)=x_{m-1}$ for $m\geq 1$. Here, we notice that $X_{s(\alpha_m)}=X_{t(\alpha_{m-1})}$ for $m\geq 1$.

In the second case, there exists a unique path $p(x)=\alpha_l\cdots \alpha_2\alpha_1$ ending at a sink such that there exist $x_m\in X_{s(\alpha_{m+1})}$ for $0\leq m\leq l-1$, and $x_l\in X_{t(\alpha_l)}$, satisfying  that $x=x_0$ and $\sigma_{\alpha_m}(x_m)=x_{m-1}$ for $1\leq m\leq l$. The length $l$ of the path $p(x)$ might be zero; this happens if and only if
$x\in X_i$ for a sink $i$.

Recall that $Q_\infty=\cup_{[p]\in \widetilde{Q}_\infty} [p]$ is a disjoint union. Then it is naturally a $Q$-algebraic branching system as in Example \ref{exm}(1). Similarly, the disjoint union $N=\cup_{i \in Q_0^s} N_i$ is  a $Q$-algebraic branching system, and so is the disjoint union $Q_\infty\cup N$.

We have the following observation, whose proof is routine.

\begin{lem}\label{lem:morphism}
Let $X$ be a perfect $Q$-algebraic branching system. Then the map $$f_X\colon X\longrightarrow Q_\infty\cup N,\;  f_X(x)=p(x)$$ is a morphism of $Q$-algebraic branching systems. \hfill $\square$
\end{lem}

We are in a position to prove Theorem \ref{thm:irre}.
\vskip 5pt

\noindent { \textit{Proof of Theorem} \ref{thm:irre}:}  The ``if" part follows from Proposition \ref{prop:FN} and Section 3. For the ``only if" part, assume that
the associated representation $\mathcal{M}(X)$ is irreducible. The morphism  in Lemma \ref{lem:morphism} induce a nonzero homomorphism $\mathcal{M}(f_X)\colon \mathcal{M}(X)\rightarrow \mathcal{M}(Q_\infty\cup N)$; it is injective, since $\mathcal{M}(X)$ is irreducible. Observe from Proposition \ref{prop:FN} that $\mathcal{M}(Q_\infty \cup N)\simeq \mathcal{F}\oplus \mathcal{N}$.

Recall from Section 3 that the representation $\mathcal{F}\oplus \mathcal{N}$ is completely reducible and any of its irreducible sub representations equals  $\mathcal{F}_{[p]}$ or $\mathcal{N}_i$ for some left-infinite path $p$ or a sink $i$. From these we infer that the image of the injective homomorphism $\mathcal{M}(f_X)$ equals $\mathcal{F}_{[p]}$ or $\mathcal{N}_i$. This implies that the image of $f_X$ equals $[p]$ or $N_i$, and then as $Q$-algebraic branching systems, $X$ is isomorphic to $[p]$ or $N_i$.
\hfill $\square$

\vskip 5pt

\section{Twisted representations}

In this section, we study the twisted representations of the irreducible representations in Section 3 under the scaling action
of the Leavitt path algebra. In particular, we obtain new irreducible representations for rational tail-equivalence classes.
In the end, we prove the faithfulness of some completely reducible  representation.

Let $Q$ be an arbitrary quiver. Denote by $k^\times$ the \emph{multiplicative group} of $k$, and by ${(k^\times)}^{Q_1}$
the product group. Its elements are of the form ${\bf a}=(a_\alpha)_{\alpha\in Q_1}$ with each $a_\alpha\in k^\times$, and its
multiplication is componentwise. For each ${\bf a}$, there is an algebra automorphism $\gamma_{{\bf a}}\colon L_k(Q)
\rightarrow L_k(Q)$ such that $\gamma_{\bf a} (e_i)=e_i$, $\gamma_{\bf a} (\alpha)=a_\alpha \alpha$, and $\gamma_{\bf a} (\alpha^*)=a_\alpha^{-1}\alpha^*$. This gives rise to a group homomorphism $\gamma\colon {(k^\times)}^{Q_1}\rightarrow {\rm Aut}(L_k(Q)) $, which is injective. This is called the \emph{(generalized) scaling action}; compare \cite[Definition 2.13]{AT}.

For an element ${\bf a}=(a_\alpha)_{\alpha\in Q_1}$ and a nontrivial path $p=\alpha_n \cdots \alpha_2\alpha_1$ in $Q$, set
$a_p=a_{\alpha_n}\cdots a_{\alpha_2} a_{\alpha_1}$. The element ${\bf a}$ is called $p$-\emph{stable} provided that
$a_p=1$.

Recall that for a representation $M$ of an algebra $A$ and an automorphism $\sigma$ of $A$,
we have the \emph{twisted representation} $M^\sigma$ as follows: $M^\sigma=M$ as vector spaces, and the action
 is given by $a. m^\sigma=(\sigma(a).m)^\sigma$.  Here, for an element $m$ in $M$, we denote by $m^\sigma$ the corresponding
element in $M^\sigma$. Moreover, the representation  $M^\sigma$ is irreducible if and only if so is $M$.

For the  Leavitt path algebra, we write the twisted representation $M^{\gamma_{\bf a}}$ simply as $M^{\bf a}$. Observe
that $M^{\bf 1}=M$.

Recall the irreducible representations $\mathcal{F}_{[p]}$ and $\mathcal{N}_i$ constructed in
Section 3. We are interested in their twisted representations $\mathcal{F}_{[p]}^{\bf a}$ and $\mathcal{N}_i^{\bf a}$.

\begin{prop}\label{prop:scaling}
Let $Q$ be a quiver, and let ${\bf a}, {\bf b}\in {(k^\times)}^{Q_1}$. We use the notation as above.
Then the following statements hold.
\begin{enumerate}
\item[(1)] For $[p]\in \widetilde{Q}_{\infty}$ an irrational class, the two representations $\mathcal{F}_{[p]}^{\bf a}$ and $\mathcal{F}_{[p]}^{\bf b}$ are isomorphic.
\item[(2)] For $[q^\infty]\in \widetilde{Q}_{\infty}$ a rational class with $q$ a simple oriented cycle,
the two representations $\mathcal{F}_{[q^\infty]}^{\bf a}$ and $\mathcal{F}_{[q^\infty]}^{\bf b}$ are isomorphic if and only if
${\bf a} {\bf b}^{-1}$ is $q$-stable.
\item[(3)] For $i\in Q_0^s$ a sink, the two representations $\mathcal{N}_i^{\bf a} $ and $\mathcal{N}_i^{\bf b}$ are isomorphic.
\end{enumerate}
\end{prop}

\begin{proof}
To show (1), it suffices to prove that $\mathcal{F}_{[p]}\simeq \mathcal{F}_{[p]}^{\bf a}$ for any ${\bf a} \in (k^\times)^{Q_1}$.
Fix $p_0\in [p]$. Then for each $q\in [p]$, we may choose  natural numbers $n$ and $m$ such that $\tau_{>n}(q)=\tau_{>m}(p_0)$. Since the left-infinite path $p_0$ is irrational, the number $n-m$ is unique for $q$. For the same reason, the scalar $\theta(q):=(a_{\tau_{\leq n}(q)})^{-1} a_{\tau_{\leq m}(p_0)}$ is independent of the choice of $n$ and $m$.  Then we have the required isomorphism $\phi\colon \mathcal{F}_{[p]}\rightarrow \mathcal{F}_{[p]}^{\bf a}$, which sends $q\in [p]$ to $\theta(q)q$. One proves (3) with
a similar argument.

To see (2), it suffices to prove that $\mathcal{F}_{[q^\infty]}\simeq \mathcal{F}_{[q^\infty]}^{\bf a}$ if and only if ${\bf a}$ is $q$-stable. For the ``only if" part, we observe that any isomorphism $\phi\colon \mathcal{F}_{[q^\infty]}\rightarrow  \mathcal{F}_{[q^\infty]}^{\bf a}$ satisfies that $\phi(q^\infty)=\lambda q^\infty$ for a nonzero scalar $\lambda$; consult
the third paragraph in the proof of Theorem \ref{thm:1}. Then $\phi(q^\infty)=\phi(q. q^\infty)=q. \phi(q^\infty)= \lambda a_q \; q^\infty$. This implies that $a_q=1$.

Consider the ``if" part.  For each $p\in [q^\infty]$, take the smallest natural number $n_0$ such that $\tau_{>n_0}(p)=q^\infty$, and
set $\theta(p)=(a_{\tau_{\leq n_0}(p)})^{-1}$; in addition, set $\theta(q^\infty)=1$. Define a linear map $\phi\colon \mathcal{F}_{[q^\infty]}\rightarrow  \mathcal{F}_{[q^\infty]}^{\bf a}$ sending $p$ to $ \theta(p) p$. It is routine to verify that this is an isomorphism of representations. Here, one needs to use that ${\bf a}$ is $q$-stable in one particular case.
\end{proof}

To summarize, we list all the irreducible representations of the Leavitt path algebra, that are constructed in this paper. For this end, we fix
for each rational class $[p]\in \widetilde{Q}^{\rm rat}_\infty$ a simple oriented cycle $q=\alpha_n \cdots \alpha_2\alpha_1$ with $p\sim q^\infty$. For each $\lambda
\in k^\times$, set ${\bf a}_{\lambda,q}=(a_\alpha)_{\alpha\in Q_1}$ such that $a_{\alpha_{1}}=\lambda$ and $a_\alpha=1$ for $\alpha\neq \alpha_1$.

Set $\mathcal{F}_{[p]}^\lambda=\mathcal{F}_{[p]}^{{\bf a}_{\lambda,q}}$. By Proposition \ref{prop:scaling}(2)
we have that for each ${\bf a}\in (k^\times)^{Q_1}$,  $\mathcal{F}^{\bf a}_{[p]}$ is isomorphic to $\mathcal{F}^{a_q}_{[p]}$; moreover, $\mathcal{F}_{[p]}^\lambda$ is isomorphic to $\mathcal{F}_{[p]}^{\lambda'}$ if and only if $\lambda=\lambda'$. Observe that $\mathcal{F}_{[p]}^1=\mathcal{F}_{[p]}$.

We obtain a list of  pairwise non-isomorphic irreducible representations for the Leavitt
path algebra $L_k(Q)$. The representations are parameterized  by the disjoint union $\widetilde{Q}^{\rm irr}_\infty  \cup (k^\times \times \widetilde{Q}^{\rm rat}_\infty) \cup Q_0^s$.

\begin{thm}\label{thm:conclusion}
Let $Q$ be a quiver and let $k$ be field. Then the following set
$$ \{\mathcal{F}_{[p]}\; |\; [p]\in \widetilde{Q}_\infty^{\rm irr}\}\;  \cup \; \{\mathcal{F}_{[p]}^\lambda\; |\;
\lambda \in k^\times, [p]\in\widetilde{Q}^{\rm rat}_\infty\} \; \cup \;  \{\mathcal{N}_i \; |\; i\in Q_0^s\}$$
consists  of pairwise non-isomorphic irreducible  representations of $L_k(Q)$.
\end{thm}

\begin{proof}
It suffices to show that these representations are pairwise non-isomorphic. This follows from
Theorem \ref{thm:1}(2), Theorem \ref{thm:N}(3) and Proposition \ref{prop:scaling}(2). Here, we need
to use the same argument therein to show that $\mathcal{F}^\lambda_{[p]}\simeq \mathcal{F}^{\lambda'}_{[p']}$ implies
that $[p]=[p']$. Moreover, $\mathcal{F}_{[p]}^\lambda$ is neither isomorphic to $\mathcal{F}_{[p']}$ with $[p']$ irrational, nor
isomorphic to $\mathcal{N}_i$ with $i$ a sink. We omit the details.
\end{proof}

We close this paper with a result on the faithfulness of the following completely reducible  representation $$\mathcal{S}=\bigoplus_{[p]\in \widetilde{Q}_\infty^{\rm irr}} \; \mathcal{F}_{[p]} \bigoplus_{\lambda \in k^\times, [p]\in\widetilde{Q}^{\rm rat}_\infty}   \; \mathcal{F}_{[p]}^\lambda \; \;  \bigoplus_{i\in Q_0^s} \; \mathcal{N}_i.$$  This partially remedies the counterexample in Remark \ref{rem:faithful}.

\begin{prop}\label{prop:faithfulness2}
Let $Q$ be a row-finite quiver, and let $k$ be an infinite field.  Then the  representation
$\mathcal{S}$ of $L_k(Q)$ is faithful.
\end{prop}

\begin{proof}
We observe that a modified argument in the proof of Proposition \ref{prop:faithfulness} will work. It suffices
to show that any nonzero element $u=\sum_{i=1}^l \lambda_i q_i$ in $L_k(Q)$ acts nontrivially on $\mathcal{S}$.
Here, $u$ is in its normal form (see (\ref{equ:1})), and $\kappa(u)=0$, that is, all the $q_i$'s are paths in $Q$. We may assume that $t(q_i)=j$
for some $j\in Q_0$ and all $1\leq i\leq l$. Without loss of generality, we assume that $q_1$ is shortest among all
the $q_i$'s.

By Proposition \ref{prop:faithfulness} and its proof, we may assume that there is a cyclic path $p=q^\infty$ starting at $j$ with $q$ a simple oriented cycle.

Consider $pq_1$ as an element in $\mathcal{F}_{[p]}^\lambda$ for some $\lambda$. Consider
$I_1=\{i \; |\; 2\leq i\leq l, q_i=q^{m_i} q_1 \mbox{ for some } m_i\geq 1\}$ and  $I_2=\{2, 3, \cdots, l\}\setminus I_1$.
Here, $l(q)m_i=l(q_i)-l(q_1)$ for $i\in I_1$. Then we have
\begin{align*} u.(pq_1) &=
\lambda_1 p + \sum_{i\in I_1} \lambda_i\;  q_i.(pq_1) + \sum_{i \in I_2} \lambda_i\;  q_i.(qq_1)\\
& = (\lambda_1 +\sum_{i\in I_1} \lambda_i \lambda^{m_i})p+ \sum_{i \in I_2} \lambda_i \; q_i.(qq_1).
 \end{align*}
 We observe that in the summation indexed by $I_2$, $q_i.(qq_1)$ is  either
zero  or a multiple of a path in $\mathcal{F}_{[p]}^\lambda$ that is different from $p$. Since the field $k$
is infinite, we  may take $\lambda\in k^\times$ such that $\lambda_1 +\sum_{i\in I_1} \lambda_i \lambda^{m_i}\neq 0$. In this
case, we have that in  $\mathcal{F}_{[p]}^\lambda$,  $u.(pq_1)\neq 0$. We are done.
\end{proof}

\begin{rem}
We point out that for a finite field $k$, the representation $\mathcal{S}$ might not be faithful. The example is
given by $Q=R_1$ in Example \ref{exm:Leavitt}, the quiver consisting of one vertex with one loop attached.
\end{rem}

\bibliography{}

\begin{thebibliography}{999}


\bibitem{AA05} {\sc G.  Abrams and G. Aranda Pino}, {\em The Leavitt path algebra of a graph,}
 J. Algebra {\bf 293} (2) (2005), 319--334.

\bibitem{AA08} {\sc G.  Abrams and G. Aranda Pino}, {\em  The Leavitt path algebras of arbitrary graphs},
 Houston J. Math. {\bf 34} (2) (2008), 423--442.


\bibitem{AAS}{\sc G. Abrams, G. Aranda Pino, and M. Siles Molina}, {\em  Finite-dimensional Leavitt path algebras},
J. Pure Appl. Algebra {\bf 209} (2007), 753--762.


\bibitem{ARS11}{\sc G. Abrams, K.M. Rangaswamy, and M. Siles Molina}, {\em The socle series of a Leavitt path algebra},
 Israel J. Math. {\bf 184} (2011), 413--435,

\bibitem{Ar} {\sc P. Ara}, {\em Finitely presented modules over Leavitt algebras,} J. Pure Appl. Algebra
{\bf 191} (1-2) (2004), 1--21.


\bibitem{AB} {\sc P. Ara and M. Brustenga,} {\em Module theory over Leavitt path algebras and K-theory,}
J. Pure Appl. Algebra {\bf 214} (2010), 1131--1151.

\bibitem{AMP} {\sc P. Ara, M.A. Moreno and E. Pardo,} {\em Nonstable K-theory for graph algebras,}
Algebr. Represent. Theory {\bf 10} (2) (2007), 157--178.

\bibitem{AMMS08} {\sc G. Aranda Pino, D. Martin Barquero, C. Martin Gonz\'{a}lez and M. Siles Molina}, {\em The socle of
a Leavitt path algebra}, J. Pure Appl. Algebra {\bf 212} (3) (2008), 500--509.

\bibitem{AMMS} {\sc G. Aranda Pino, D. Martin Barquero, C. Martin Gonz\'{a}lez and M. Siles Molina}, {\em Socle theory
for Leavitt path algebras of arbitrary graphs}, Rev. Mat. Iberoam. {\bf 26} (2) (2010), 611--638.


\bibitem{ARS}{\sc M. Auslander, I. Reiten, and S.O. Smal{\o},}
Representation Theory of Artin Algebras, Cambridge Studies in Adv.
Math. {\bf 36}, Cambridge Univ. Press, Cambridge, 1995.

\bibitem{AT} {\sc G. Abrams and M. Tomforde}, {\em Isomorphism and Morita equivalence of graph algebras}, Trans. Amer.
Math. Soc. {\bf 363} (7) (2011), 3733--3767.


\bibitem{Bon} {\sc A. Bonderenko}, {\em Description of Anti-Fock representation set of $\ast$-algebras generated
by the relation $XX^*=f(X^*X)$}, Proc. Inst. Math. NAS Ukraine {\bf 50} (3) (2004), 1053--1060.

\bibitem{Ch} {\sc X.W. Chen}, {\em The singularity category of an algebra with radical square zero}, Doc. Math.
{\bf 16} (2011), 921--936


\bibitem{CK} {\sc J. Cuntz and W. Krieger}, {\em A class of $C^{\ast}$-algebras and topological Markov chains,}
Invent. Math. {\bf 56} (3) (1980), 251--268.


\bibitem{GR} {\sc D. Goncalves and D. Royer}, {\em On the representations of Leavitt path algebras}, J. Algebra {\bf 333} (2011), 258--272.

\bibitem{G09} {\sc K.R. Goodearl}, {\em Leavitt path algebras and direct limits}, in: Rings, Modules and Representations, 165--187,
Contemp. Math. {\bf 480}, Amer. Math. Soc., Providence, RI, 2009.



\bibitem{Lev57} {\sc W.G. Leavitt,} {\em Modules without invariant basis number,} Proc. Amer. Math. Soc. {\bf 8} (1957),
322--328.


\bibitem{Lev} {\sc W.G.  Leavitt,} {\em The module type of a ring,} Trans. Amer. Math. Soc. {\bf 103} (1962), 113--130.


\bibitem{LM} {\sc D. Lind and B. Marcus}, An Introduction to Symbolic Dynamics and Codings, Cambridge Univ. Press, Cambridge, 1995.


\bibitem{LNS} {\sc J. Lowrynowciz, K. Nouno,  and O. Suzuki,} {\em A duality theorem for fractal sets and Cuntz algebras and
their central extensions}, Kyoto Univ. preprint, 2003. Available at http://hdl.handle.net/2433/43311.



\bibitem{Ra05} {\sc I.  Raeburn,} Graph Algebras, CBMS Regional Conference Series in Mathematics {\bf 103}, Amer. Math. Soc., Providence, RI, 2005.



\bibitem{Smii} {\sc S.P. Smith}, {\em The space of Penrose tilings and the non-commutative curve with homogeneous coordinate ring $k\langle x, y \rangle /{(y^2)}$,} arXiv:1104.3811.


\bibitem{Smi} {\sc S.P. Smith}, {\em Category equivalences involving graded modules over path algebras of quivers,} arXiv:1107.3511.


\bibitem{Tom} {\sc M. Tomforde,} {\em Uniqueness theorems and ideal structure for Leavitt path algebras, } J. Algebra {\bf 318} (2007), 270--299.




\end{thebibliography}

\vskip 10pt

 {\footnotesize \noindent Xiao-Wu Chen\\
  School of Mathematical Sciences,
  University of Science and Technology of
China, Hefei 230026, Anhui Province, PR China \\
Wu Wen-Tsun Key Laboratory of Mathematics, USTC, Chinese Academy of Sciences, Hefei 230026, Anhui Province, PR China.\\
URL: http://mail.ustc.edu.cn/$^\sim$xwchen}

\end{document}